\documentclass[10pt]{article}   

\usepackage{amsmath}
\usepackage{amsfonts}
\usepackage{amssymb}
\usepackage{amsthm}
\usepackage{amssymb}

\theoremstyle{plain}
\newtheorem{theorem}{Theorem}
\newtheorem{proposition}[theorem]{Proposition}
\newtheorem{lemma}[theorem]{Lemma}
\newtheorem{corollary}[theorem]{Corollary}

\theoremstyle{definition}

\theoremstyle{remark}
\newtheorem{remark}[theorem]{Remark}

\begin{document}
\begin{center}
\vspace{3cm}

{\Large {\bf ON REAL HYPERSURFACES IN NON-FLAT COMPLEX SPACE FORMS WITH A CONDITION ON THE STRUCTURE JACOBI OPERATOR}}
\bigskip

S.H. Kon, Tee-How Loo, Shiquan Ren

{\it Institute of Mathematical Sciences, University of Malaya, 50603 Kuala Lumpur, Malaysia}

\end{center}
\bigskip

{\small {\bf Abstract.} In this paper we prove some classification theorems of real hypersurfaces in $M_n(c)$ satisfying certain conditions on the covariant derivative of the structure Jacobi operator. We also prove the non-existence of real hypersurfaces with Codazzi type structure Jacobi operator in $M_n(c)$. }\bigskip

\bigskip

\bigskip
\section{Introduction}

The study of real hypersurfaces of Kaehlerian manifolds has been an important subject in the geometry of submanifolds, especially when the ambient space is a complex space form.
Complex space forms are the simplest Kaehlerian manifolds, and they are the complex case analogues of real space forms.
The complex structure of the ambient space induced an almost contact structure on its real hypersurfaces.
The interaction between the almost contact structure and shape operator results in interesting properties of real hypersurfaces in complex space forms (for instance, see \cite{jdg, montiel-romero, okumura, okumura2, okumura3, vernon}).

Let $M_{n}(c)$ be an $n$-dimensional non-flat complex space form with constant holomorphic sectional curvature $4c$.
It is known that a complete and simply connected non-flat complex space form is either a complex projective space $(c > 0)$, denoted by $CP^n$, or a complex hyperbolic space $(c < 0)$, denoted by $CH^n$.
Throughout this paper, all manifolds, vector fields, etc., will be considered of class $C^\infty$ unless
otherwise stated and it is always assumed $c=\pm 1$ and $n>2$.

We also assume that $M$ is a connected real hypersurface in $M_n(c)$, without boundary.
Since we only study local properties, we do not distinguish between $M$ and an open subset of $M$.
Let $N$ be a locally defined unit normal vector field on $M$.
Denote by $\nabla$ the Levi-Civita connection on $M$ induced from $M_n(c)$.
Let $\langle~, ~\rangle$ denote the Riemannian metric of $M$ induced from the Riemannian
metric of $M_{n}(c)$ and $A$ the shape operator of $M$ in $M_{n}(c)$.
We define a tensor field $\phi$ of type (1,1), a vector field $\xi$ and a 1-form $\eta$ by
\[
JX = \phi X + \eta (X)N,    \quad JN = - \xi,
\]
\noindent
where $X\in\Gamma(TM)$ and $\eta(X)=\langle X,\xi\rangle$. 
Let $D$, called the holomorphic distribution, denote the distribution on $M$ given by all vectors orthogonal to $\xi$ at each point of $M$.  Let $\alpha=\langle A\xi,\xi\rangle$.
A real hypersurface $M$ in $M_n(c)$ is said to be Hopf if $\xi$ is principal, i.e., $A\xi=\alpha\xi$; and
$M$ is said to be totally $\eta$-umbilical if there exists a function $\lambda$ such that $AX=\lambda X$ for all $X\in\Gamma(D)$. It has been proved that $\lambda$ must be a constant (cf. \cite{focal}, \cite{montiel}).
The following celebrated theorems classify certain families of well-behaved Hopf hypersurfaces in $CP^n$ and $CH^n$ respectively.
\begin{theorem}
(\cite{34})
\label{thcpn}
Let $M$ be a Hopf hypersurface with constant principal curvatures in $CP^n$. Then $M$ is locally congruent to a tube of radius $r$ over one of the following Keahlerian manifolds:

($A_1$) hyperplane $CP^{n-1}$, where $0<r<\pi/2$;

($A_2$) totally geodesic $CP^p$ $(0< p<n-1)$, where $0<r<\pi/2$;

(B) complex quadric $Q_{n-1}$, where $0<r<\pi/4$;

(C) $CP^1\times CP^{(n-1)/2}$, where $0<r<\pi/4$ and $n>4$ is odd;

(D) complex Grassmann $G_{2,5}$, where $0<r<\pi/4$ and $n=9$;

(E) Hermitian symmetric space $SO(10)/U(5)$, where $0<r<\pi/4$ and $n=15$.

\end{theorem}
\begin{remark}
In $CP^n$, a real hypersurface is a geodesic hypersphere of radius $r$  if and only if it is a tube of radius $\frac{\pi}{2}-r$ over ${C}P^{n-1}$, where $0<r<\pi/2$.
\end{remark}
\begin{theorem}(\cite{3})
\label{thchn}
Let $M$ be a Hopf hypersurface with constant principal curvatures in $CH^n$. Then $M$ is locally congruent to one of the following:

($A_0$) a horosphere;

($A_1$) a geodesic hypersphere $CH^0$ or a tube over a hyperplane $CH^{n-1}$;

($A_2$) a tube over a totally geodesic $CH^p$ $0<p<n-1$;

(B) a tube over a totally real hyperbolic space $RH^n$.

\end{theorem}



In the following, by a real hypersurface of type $A$, we mean of type $A_1$, $A_2$ (resp. of $A_0$, $A_1$, $A_2$) for $c>0$ (resp. $c<0$). By using the above theorems, it is possible to obtain various classification theorems under conditions on certain symmetric operators, such as the shape operator, Ricci tensor, and the structure Jacobi operator.

The Jacobi operator $R_X$ with respect to a tangent vector field $X$ on an open subset of $M$, is defined by
$R_X(Y)=R(Y,X)X$, for $Y\in\Gamma(TM)$, where $R$ is the curvature tensor on $M$. 
In particular,
$R_\xi$ is called the structure Jacobi operator of $M$. Various conditions on the structure Jacobi operator have been considered for either characterizing certain classes of real hypersurfaces or non-existence problems. The study of real hypersurfaces with parallel $R_{\xi}$ was initiated in \cite{rocky}. In \cite{1}, a characterization of ruled real hypersurfaces in $M_n(c)$ under the condition $(\nabla_XR_\xi)\xi=0$ for all $X\in\Gamma(D)$ was given, and the non-existence of real hypersurfaces in $M_n(c)$ with $D$-parallel and $D$-recurrent $R_\xi$ can be deduced due to this result. In \cite{ki}, real hypersurfaces in $M_n(c)$ with cyclic-parallel $R_{\xi}$ have been studied; and in \cite{lie}, real hypersurfaces in $CP^n$ with Lie $\xi$-parallel $R_\xi$ were considered.

In \cite{ricci} and \cite{maeda}, the following conditions for the shape operator $A$ and the Ricci tensor $S$
\begin{eqnarray*}
(\nabla_XA)Y=-c\{\langle\phi X,Y\rangle\xi+\eta(Y)\phi X\},\nonumber\\
(\nabla_XS)Y=k\{\langle\phi X,Y\rangle\xi+\eta(Y)\phi X\}
\end{eqnarray*}
have been studied respectively for real hypersurfaces in $CP^n$. In \cite{nihonkai}, these two conditions were considered for real hypersurfaces in $CH^n$. 
The study of such conditions was motivated by the non-existence of real hypersurfaces in $M_n(c)$ with either parallel shape operator or parallel Ricci tensor (cf. page~243 and page~271 of \cite{Ryan}). On the other hand, there does not exist any real hypersurfaces in $M_n(c)$ with parallel $R_{\xi}$ as proved in \cite{rocky}.  It is natural to study a similar condition on $R_{\xi}$ in the following theorem.
\begin{theorem}\label{th3}
Let $M$ be a real hypersurface in $M_n(c)$. Then $M$ satisfies
\begin{equation}
(\nabla_XR_{\xi})Y=k(\langle\phi X,Y\rangle\xi+\eta(Y)\phi X),
\label{umb}
\end{equation}
for all $X,Y\in \Gamma(TM)$ if and only if $M$ is totally $\eta$-umbilical, i.e., it is locally congruent to one of the following real hypersurfaces:

for $c>0$,

(a). geodesic hyperspheres in $CP^n$;

for $c<0$,

(a). geodesic hyperspheres in $CH^n$;

(b). tubes around complex hyperbolic hyperplane in $CH^n$;

(c). horospheres in $CH^n$.

Furthermore, we have $k\neq 0$.
\end{theorem}
Let $F$ be a tensor field of type $(1,1)$ on $M$. $F$ is said to be of Codazzi type if for any $X,Y\in\Gamma(TM)$,
\begin{equation}
\label{40}
(\nabla_XF)Y=(\nabla_YF)X.
\end{equation}
The condition (\ref{40}) is weaker than the parallelism of $F$, and it is natural since for a totally geodesic hypersurface in a Riemannian manifold (if it exists), the shape operator is of Codazzi type. A Riemannian manifold is said to have harmonic curvature if its Ricci tensor is of Codazzi type. This cannot happen for a Hopf hypersurface in $M_n(c)$ (cf. \cite[page 279]{Ryan}). Moreover, in \cite{codazzi}, the non-existence of real hypersurfaces in $CP^n$ with Codazzi type structure Jacobi operator has been obtained. We will generalize this statement in Theorem~\ref{th2}.


\begin{theorem}\label{th2}
 There does not exist any real hypersurface $M$ in $M_n(c)$ with its structure Jacobi operator of Codazzi type.
\end{theorem}

In order to prove Theorem~\ref{th3} and Theorem~\ref{th2} simultaneously, we consider a generalized condition in Theorem~\ref{th0}.

\begin{theorem}
\label{th0}
A real hypersurface $M$ in $M_n(c)$ satisfies
\begin{align}
\langle(\nabla_XR_{\xi})Y-(\nabla_YR_{\xi})X,W\rangle
=&k(2\eta(W)\langle\phi X,Y\rangle+\eta(Y)\langle\phi X,W\rangle  \nonumber\\
 & -\eta(X)\langle\phi Y,W\rangle)
\label{1}
\end{align}
for all $X,Y,W\in\Gamma(TM)$,
if and only if $M$ is a totally $\eta$-umbilical real hypersurface, or an arbitrary Hopf hypersurface with $\alpha=k=1$, $c=-1$. Furthermore, we have $k\neq 0$.
\end{theorem}

 Finally, Theorem~\ref{th3} and Theorem~\ref{th0} give two equivalent characterizations for totally $\eta$-umbilical real hypersurfaces in $CP^n$, as stated in the following corollary.

\begin{corollary}
For a real hypersurface $M$ in $CP^n$, the following conditions are equivalent:

(a). (\ref{umb}) holds for all $X,Y,W\in\Gamma(TM)$;

(b). (\ref{1}) holds for all $X,Y\in\Gamma(TM)$;

(c). $M$ is totally $\eta$-umbilical.
\end{corollary}


\section{Preliminaries}
 Let $M$ be a connected real hypersurface in
$M_n(c)$ without boundary. The set of tensors $(\phi, \xi, \eta, \langle ,
\rangle)$ is an almost contact metric structure on $M$, i.e., they
satisfy the following
\begin{equation*}
    \phi ^{2}X = -X + \eta (X) \xi,
    \quad \phi \xi = 0,
    \quad \eta (\phi X) = 0,
    \quad \eta (\xi ) = 1.
\end{equation*}
From the parallelism of $J$, we get
\begin{equation*}
    ({\nabla}_X{\phi})Y=\eta(Y)AX-{\langle}AX,Y\rangle\xi,
\end{equation*}
\begin{equation*}
 \nabla_X\xi={\phi}AX.
\end{equation*}
\noindent Let $R$ be the curvature tensor of $M$. Then the Gauss and
Codazzi equations are respectively given by
\begin{eqnarray*}
R(X,Y)Z=c \{\langle Y,Z \rangle X- \langle X,Z \rangle\ Y+ \langle \phi Y,Z \rangle \phi X-\langle \phi X,Z \rangle \phi Y  \\
         - 2\langle \phi X,Y \rangle \phi Z \} + \langle AY,Z \rangle AX-
                    \langle AX,Z \rangle AY ,
\end{eqnarray*}
\noindent
\begin{eqnarray*}
({\nabla_X}A)Y-({\nabla_Y}A)X=c\{\eta(X){\phi}Y-\eta(Y){\phi}X-2\langle{\phi}X,Y\rangle\xi\}.
\end{eqnarray*}

From the Gauss equation, we have
\begin{equation}
R_{\xi}Y=c\{Y-\eta(Y)\xi\}+{\alpha}AY-\eta(AY)A\xi,
 \label{g1}
\end{equation}
\begin{eqnarray}
(\nabla_XR_{\xi})Y=-c{\langle}Y,{\phi}AX\rangle\xi-c{\eta}(Y){\phi}AX+(X\alpha)AY+\alpha({\nabla_X}A)Y
\nonumber \\
 -\langle(\nabla_XA)Y,\xi{\rangle}A\xi-{\langle}Y,A{\phi}AX{\rangle}A\xi
 \nonumber \\
 -\eta(AY)(\nabla_XA)\xi-\eta(AY)A{\phi}AX,
 \label{g2}
\end{eqnarray}
for any $X,Y\in\Gamma(TM)$.

For a Hopf hypersurface $M$ in $M_n(c)$, it can be proved that $\alpha$ is a constant (cf. \cite{jdg}). The following cited results will be used in the proof of our theorems.

\begin{theorem}(
\cite[page 262-264]{Ryan})
\label{lemma3}
Let $M$ be a real hypersurface in $M_n(c)$. Then $M$ is locally congruent to a Hopf hypersurface of type $A$ if and only if
\[
(\nabla_XA)Y=-c\{\langle\phi X,Y\rangle\xi+\eta(Y)\phi X\}
\]
for any $X,Y\in\Gamma(TM)$.
\end{theorem}
\begin{theorem}
(\cite{focal}, \cite{montiel})
\label{classifyumb}
Let $M$ be a real hypersurface in $M_n(c)$. Then $M$ is totally $\eta$-umbilical if and only if it is one of the following:

for $c>0$,

(a). geodesic hyperspheres in $CP^n$;

for $c<0$,

(a). geodesic hyperspheres in $CH^n$;

(b). tubes around complex hyperbolic hyperplane in $CH^n$;

(c). horospheres in $CH^n$.
\end{theorem}

The corresponding principal curvatures of the real hypersurfaces in Theorem~\ref{classifyumb} are in Table~\ref{x} ($\lambda$ is the principal curvature such that $AX=\lambda X$ for $X\in\Gamma(D)$).

\begin{table}[h]

\caption{} 
\centering 
\begin{tabular}{|c c c c c|} 
\hline 
 & Case & Radius & $\alpha$ & $\lambda$   
  \\ [0.5ex]
\hline 
 $c>0$ & (a) & $r$  & $2\cot 2r$ & $\cot r$  \\

\hline
  $c<0$ & (a)& $r$ & $2\coth 2r$ & $\coth r$  \\
 $c<0$ & (b)& $r$ & $2\coth 2r$ & $\tanh r$  \\
 $c<0$ & (c)& - & $2$ & $1$ \\ [1ex] 
\hline 
\end{tabular}

\label{x} 
\end{table}

Let $\beta=||\phi A\xi||$. If $M$ is a non-Hopf real hypersurface in $M_n(c)$ then $\beta>0$. We can define a unit vector field $U$ in $\Gamma(D)$ by $U=-\dfrac{1}{\beta}\phi^2A\xi$ and a distribution $D_U$ by
\begin{eqnarray*}
D_U=\{X\in T_xM|X\perp \xi, U, \phi U\},  x\in M.
\end{eqnarray*}
\begin{lemma}
\label{lemma1}(\cite{1})
Let $M$ be a non-Hopf real hypersurface in $M_n(c)$. Suppose $M$ satisfies the following:

(a) $A\xi=\alpha\xi+\beta U$, $AU=\beta\xi+\gamma U$ and $A\phi U=\delta\phi U$ for some functions $\gamma$ and $\delta$ on $M$;

(b) there exists a unit vector field $Z\in\Gamma(D_U)$ such that $AZ=\lambda Z$ and $A\phi Z=\lambda\phi Z$ for some function $\lambda$ on $M$.

 Then we have
\begin{equation}
\label{correct2}
\beta\lambda(\lambda-\delta)-(\lambda-\gamma)\phi U\lambda=0.
\end{equation}
\end{lemma}
\begin{proof}
Suppose $M$ is such a real hypersurface satisfying (a) and (b). Taking inner product in the Codazzi equation
\begin{equation*}
(\nabla_ZA)\xi-(\nabla_{\xi} A)Z=-c\phi Z
\end{equation*}
with $\phi Z$, we obtain
\begin{equation*}
\beta\langle\nabla_ZU,\phi Z\rangle=\lambda^2-\alpha\lambda-c.
\end{equation*}
Taking inner product in the Codazzi equation
\begin{equation*}
(\nabla_ZA)\phi U-(\nabla_{\phi U} A)Z=0
\end{equation*}
with $Z$, we obtain
\begin{equation*}
(\delta-\lambda)\langle\nabla_Z\phi U, Z\rangle=\phi U\lambda.
\end{equation*}
By using
\begin{equation*}
\nabla_Z\phi U=\phi\nabla_ZU,
\end{equation*}
we have
\begin{equation}
(\lambda-\delta)(\lambda^2-\alpha\lambda-c)=\beta\phi U\lambda.
\label{correct1}
\end{equation}
Taking inner product in the Codazzi equation $(\nabla_ZA)\phi Z-(\nabla_{\phi Z}A)Z=-2c\xi$ with $\xi$ and $U$ respectively, we obtain
\begin{equation*}
\langle\nabla_{\phi Z}Z-\nabla_Z {\phi}Z,U\rangle=\frac{2(\lambda^2-\alpha\lambda-c)}{\beta},
\end{equation*}
\begin{equation*}
(\lambda-\gamma)(\langle\nabla_{{\phi}Z}Z,U\rangle-\langle\nabla_Z\phi Z,U\rangle)=2\beta\lambda.
\end{equation*}
Combining these two equations, we obtain
\begin{equation}
(\lambda-\gamma)(\lambda^2-\alpha\lambda-c)-\beta^2\lambda=0.
\label{correct}
\end{equation}
From (\ref{correct1}) and (\ref{correct}), we get (\ref{correct2}).
\end{proof}

\begin{lemma}
\label{lemma100}(\cite{jdg})
Let $M$ be a Hopf hypersurface in $M_n(c)$ with $A\xi=\alpha\xi$. Then
\begin{equation*}
2A\phi A-\alpha(A\phi+\phi A)-2c\phi=0.
\end{equation*}
\end{lemma}

\section{A Lemma}
In this section, We shall obtain a lemma that is useful for the proof of the non-Hopf case of our theorems. The generalized condition in this lemma enables it to be applied in a number of problems.

\begin{lemma}
\label{lemma2}
Let $M$ be a non-Hopf real hypersurface in $M_n(c)$. Suppose $M$ satisfies  $A\xi=\epsilon c\xi+{\beta}U$,
$AU=\beta\xi+\epsilon(\beta^2-c)U$, $A{\phi}U=-\epsilon c{\phi}U$, where $U$ is a unit vector field in
$D$, $\beta$ is a nonvanishing function defined on $M$ and $\epsilon=\pm 1$. Then $c>0$ and there exists a vector field $X\in\Gamma(D_U)$ such that $AX\neq -\epsilon X$.
\end{lemma}
\begin{proof}
Suppose $M$ is such a real hypersurface. Taking inner product in the Codazzi equation $(\nabla_UA)\phi U-(\nabla_{\phi U}A)U=-2c\xi$ with $U$ and $\xi$ respectively, we obtain
\begin{equation}
-\beta\langle\nabla_U\phi U,U\rangle+\beta^2-3c-2\phi U\beta=0,
\label{6.14.1}
\end{equation}
\begin{equation*}
-\beta\langle\nabla_U\phi U,U\rangle+3c\beta^2-4c^2+2c-\phi U\beta=0.
\end{equation*}
From these two equations, we obtain
\begin{equation}
\label{641}
\beta^2-3c\beta^2+4c^2-5c-\phi U\beta=0.
\end{equation}
Taking inner product in the Codazzi equation
\begin{equation}
\label{100}
(\nabla_{\phi U}A)\xi-(\nabla_{\xi}A)\phi U=cU
\end{equation}
with $U$, $\xi$, $\phi U$ respectively, we obtain
\begin{equation}
\epsilon\beta^2\langle\nabla_{\xi}\phi U,U\rangle+2c^2-c\beta^2-\beta^2-c+\phi U\beta=0,
\label{zzz}
\end{equation}
\begin{equation}
\langle\nabla_{\xi}\phi U,U\rangle-4\epsilon c=0,
\label{6.14.4}
\end{equation}
\begin{equation}
\label{a.6}
\langle\nabla_{\phi U}U,\phi U\rangle=0.
\end{equation}
From (\ref{zzz}) and (\ref{6.14.4}), we obtain
\begin{equation}
\label{642}
3c\beta^2+2c^2-\beta^2-c+\phi U\beta=0.
\end{equation}
By summing up (\ref{641}) and (\ref{642}), we obtain
\begin{equation*}
c(c-1)=0,
\end{equation*}
which cannot happen when $c=-1$. Hence $c=1$ and (\ref{641}) becomes
\begin{equation}
\label{a.3}
\phi U\beta=-2\beta^2-1.
\end{equation}
By substituting (\ref{a.3}) into (\ref{6.14.1}), we have
\begin{equation}
\label{a.4}
\beta\langle\nabla_U\phi U,U\rangle=5\beta^2-1.
\end{equation}
From (\ref{6.14.4}), we have
\begin{equation}
\label{a.5}
\langle\nabla_{\xi}\phi U,U\rangle=4\epsilon.
\end{equation}
In order to prove the second assertion, we shall suppose to the contrary that for any $Z\in\Gamma(D_U)$, $AZ=-\epsilon Z$.
Take inner product in the Codazzi equation $(\nabla_ZA)\xi-(\nabla_{\xi}A)Z=-\phi Z$ with $U$, $\phi U$, $\xi$ respectively,
\begin{equation}
\label{a.7}
Z\beta+\epsilon\beta^2\langle\nabla_{\xi}Z,U\rangle=0,
\end{equation}
\begin{equation}
\label{a.8}
\langle\nabla_Z U,\phi U\rangle=0,
\end{equation}
\begin{equation}
\label{a.9}
\langle\nabla_{\xi}Z,U\rangle=0.
\end{equation}
(\ref{a.7}) and (\ref{a.9}) imply $Z\beta=0$. Taking inner product in the Codazzi equation $(\nabla_ZA)U-(\nabla_UA)Z=0$ with $U$ and with the help of $Z\beta=0$, we have
\begin{equation}
\label{a.12}
\langle\nabla_UZ,U\rangle=0.
\end{equation}
Taking inner product with $Z$ in (\ref{100}), we have
\begin{equation}
\label{a.11}
\langle\nabla_{\phi U}U,Z\rangle=0.
\end{equation}
From (\ref{a.6}), (\ref{a.11}) and $\langle\nabla_{\phi U}U,\xi\rangle=-\epsilon$, we obtain
\begin{equation}
\label{a.13}
\nabla_{\phi U}U=-\epsilon\xi.
\end{equation}
Hence
\begin{equation}
\label{a.14}
\nabla_{\phi U}\phi U=(\nabla_{\phi U}\phi)U+\phi\nabla_{\phi U}U=0.
\end{equation}
From (\ref{a.4}), (\ref{a.12}) and $\langle\nabla_UU,\xi\rangle=0$, we obtain
\begin{equation}
\label{a.15}
\nabla_UU=\frac{1-5\beta^2}{\beta}\phi U.
\end{equation}
Hence
\begin{equation}
\label{a.16}
\nabla_U\phi U=\epsilon(1-\beta^2)\xi+\frac{5\beta^2-1}{\beta}U.
\end{equation}
From (\ref{a.5}), (\ref{a.9}) and $\langle\nabla_{\xi}U,\xi\rangle=0$, we obtain
\begin{equation}
\label{a.17}
\nabla_{\xi} U=-4\epsilon\phi U.
\end{equation}
Finally, we also have
\begin{equation}
\label{a.18}
\nabla_{\phi U}\xi=\phi A\phi U=\epsilon U.
\end{equation}
Let $X=U$, $Y=\phi U$ and $Z=U$ in the Gauss equation. Then we have
\begin{equation}
R(U,\phi U)U=(\beta^2-5)\phi U.
\label{a.19}
\end{equation}
On the other hand, it follows from (\ref{a.3}), (\ref{a.13})--(\ref{a.18}) and
\begin{equation*}
R(U,\phi U)U=\nabla_U\nabla_{\phi U}U-\nabla_{\phi U}\nabla_UU-\nabla_{[U,\phi U]}U
\end{equation*}
that
\begin{equation}
R(U,\phi U)U=(10\beta^2-8)\phi U.
\label{a.20}
\end{equation}
From (\ref{a.19}) and (\ref{a.20}), we see that $\beta$ is a constant. This contradicts (\ref{a.3}).
\end{proof}

\section{Auxiliary propositions}
In this section we mainly focus on the Hopf case and prove some propositions about Hopf hypersurfaces in preparation for the proof of our theorems.
\begin{proposition}
\label{prop101}
Let $M$ be a real hypersurface in $M_n(c)$. If $M$ is totally $\eta$-umbilical then $M$ satisfies (\ref{umb}). Furthermore, $k\neq 0$.
\end{proposition}
\begin{proof}
For a totally $\eta$-umbilical real hypersurface $M$, we have
\begin{equation}
AX=\lambda X+(\alpha-\lambda)\eta(X)\xi
\label{2.1}
\end{equation}
for any $X\in\Gamma(TM)$, where $\lambda$ and $\alpha$ are two constants. Applying (\ref{2.1}) and Theorem~\ref{lemma3} 
 to (\ref{g2}), we have
\begin{eqnarray}
(\nabla_XR_{\xi})Y=-\lambda(c+\alpha\lambda)\{\langle Y,\phi X\rangle\xi+\eta(Y)\phi X\}.
\label{oct16}
\end{eqnarray}
Hence $M$ satisfies (\ref{umb}). From Lemma~\ref{lemma100}, we have $\lambda^2=\alpha\lambda+c$, hence $\lambda\neq 0$. Then from the right-hand side of (\ref{oct16}), we see that $k=-\lambda(c+\alpha\lambda)=-\lambda^3\neq 0$.
\end{proof}






\begin{proposition}
\label{105}
Let $M$ be a Hopf hypersurface in $CH^n$ with $\alpha=1$. Then $M$ satisfies (\ref{1}) for $k=1$ and does not satisfy (\ref{umb}) for $k=1$.
\end{proposition}
\begin{proof}
We suppose $k=1$. Then (\ref{g2}) reduces to
\begin{eqnarray}
(\nabla_XR_{\xi})Y=\langle Y,\phi AX\rangle\xi+\eta(Y)\phi AX+(\nabla_XA)Y
-\langle Y,(\nabla_XA)\xi\rangle\xi
\nonumber\\
-\langle Y,A\phi AX\rangle\xi-\eta(Y)(\nabla_XA)\xi-\eta(Y)A\phi AX.
\label{oct5-1}
\end{eqnarray}
By a direct computation, we have
$(\nabla_XA)\xi=\phi AX-A\phi AX$. Substituting this equation into (\ref{oct5-1}), we obtain
\begin{eqnarray}
(\nabla_XR_{\xi})Y=(\nabla_XA)Y.
\label{last}
\end{eqnarray}
By using (\ref{last}) and the Codazzi equation, we see that $M$ satisfies (\ref{1}) for $k=1$.

Next, suppose there exists a Hopf hypersurface $M$ satisfying (\ref{umb}) for $k=1$. Then by (\ref{last}), (\ref{umb}) becomes
\begin{equation}
(\nabla_XA)Y=\langle\phi X,Y\rangle\xi+\eta(Y)\phi X.
\label{2.2}
\end{equation} By Theorem~\ref{lemma3}, $M$ is locally congruent to a Hopf hypersurface of type $A$ with $\alpha=1$. According to \cite[page 254-257, 260]{Ryan}, $\alpha\neq1$ for real hypersurfaces of type $A$. This is a contradiction. Hence we conclude that $M$ does not satisfy (\ref{umb}) for $k=1$.
\end{proof}

Let $M$ be a real hypersurface in $M_n(c)$ satisfying (\ref{1}). Then by applying (\ref{g2}) and the Codazzi equation, (\ref{1}) becomes
\begin{eqnarray}
0=\langle(\alpha c+k)(\eta(X)\phi Y-\eta(Y)\phi X-2\langle\phi X,Y\rangle\xi)+2\langle AX,\phi AY\rangle A\xi \nonumber \\
+c(\langle X,(A\phi+\phi A)Y\rangle\xi+2\langle\phi X,Y\rangle A\xi-\eta(Y)\phi AX+\eta(X)\phi AY) \nonumber \\
+(X\alpha)AY-(Y\alpha)AX+\eta(AX)(\nabla_YA)\xi-\eta(AY)(\nabla_XA)\xi \nonumber \\ +\eta(AX)A\phi AY-\eta(AY)A\phi AX,W\rangle
\label{2}
\end{eqnarray}
for any $X,Y,W\in\Gamma(TM)$.

\begin{proposition}
\label{103}
Let $M$ be a Hopf hypersurface in $M_n(c)$ satisfying (\ref{1}). Then either

(1). $M$ is a totally $\eta$-umbilical real hypersurface;

or

(2). $\alpha=k=1$, $c=-1$.
\end{proposition}
\begin{proof}

Recall that $\alpha$ is a constant when $M$ is Hopf.
Let $X=\xi$, $Y,W\in\Gamma(D)$ in (\ref{2}), we have
\begin{equation}
\label{hopf1}
(\alpha c+k)\langle\phi Y,W\rangle+(\alpha^2+c)\langle\phi AY,W\rangle=0.
\end{equation}
Let $Y\in\Gamma(D)$ be a unit principal vector field
with $AY=\lambda Y$ and let $W=\phi Y$ in (\ref{hopf1}). Then
\begin{equation}
\label{22}
(\alpha^2+c)\lambda+(\alpha c+k)=0.
\end{equation}
We consider two cases: $\alpha^2+c\neq 0$, $\alpha^2+c=0$.

{\bf {\sc Case-i}}. {\emph{$\alpha^2+c\neq 0$.}}

Since $Y$ is an arbitrary unit principal vector field
 in $\Gamma(D)$, we have $AX=\lambda X$ for all $X\in\Gamma(D)$, where
\begin{eqnarray*}
\lambda=-\dfrac{\alpha c+k}{\alpha^2+c}.
\end{eqnarray*}
Therefore, $M$ is totally $\eta$-umbilical.

{\bf {\sc Case-ii}}. {\emph{$\alpha^2+c=0$.}}

Then $c=-1$, and by replacing the unit normal vector
field $N$ with $-N$ if necessary, we have $\alpha=1$. By (\ref{22}), we have
 $k=\alpha=1$.
\end{proof}

\section{Proof of the theorems}
In this section we first prove the non-existence of non-Hopf
 real hypersurface satisfying condition (\ref{umb}) or (\ref{1}), then prove the theorems.



Let $M$ be a real hypersurface in $M_n(c)$. Letting $W=\xi$ in (\ref{2}), we have
\begin{eqnarray}
2k\langle\phi X,Y\rangle+2\alpha\langle A\phi AX,Y\rangle+c\langle(\phi A+A\phi)X,Y\rangle \nonumber \\ -\eta(AY)\eta(A\phi AX)+\eta(AX)\eta(A\phi AY)=0
\label{8}
\end{eqnarray}
for any $X,Y\in\Gamma(TM)$. From (\ref{8}), we have
\begin{eqnarray}
2k\phi X+2\alpha A\phi AX+c(\phi A+A\phi)X-\eta(A\phi AX)A\xi-\eta(AX)A\phi A\xi=0
\label{9}
\end{eqnarray}
for any $X\in\Gamma(TM)$.

 \begin{proposition}
\label{102}
There does not exist any non-Hopf real hypersurface $M$ in $M_n(c)$ satisfying (\ref{1}).
\end{proposition}
\begin{proof}

Suppose $M$ is a non-Hopf real hypersurface satisfying (\ref{1}) with $A\xi=\alpha\xi+\beta U$, $\beta$
a non-vanishing function and $U\in\Gamma(D)$ a unit vector field. 


Let $X=\xi$ in (\ref{9}). Then we have $\alpha\neq 0$ and
\begin{eqnarray}
A\phi U=-\frac{c}{\alpha}\phi U.
\label{10}
\end{eqnarray}
Let $X=\phi U$ in (\ref{9}). Then with the help of (\ref{10}) we obtain
\begin{eqnarray}
AU=\beta\xi+(\frac{2k}{c}+\frac{\beta^2-c}{\alpha})U.
\label{11}
\end{eqnarray}
Hence $D_U$ is $A$-invariant. 



Let $X=\phi U$, $Y\in\Gamma(D_U)$, $W\in\Gamma(TM)$ in (\ref{2}). Then we have
\begin{equation}
\label{non1}
(\phi U\alpha) AY+\frac{c}{\alpha}(Y\alpha)\phi U=0.
\end{equation}
Taking inner product with $\phi U$ in (\ref{non1}), we have $Y\alpha=0$.

Let $X=U$, $Y\in\Gamma(D_U)$, $W\in\Gamma(TM)$ in (\ref{2}). Then with the help of $Y\alpha=0$, we have
\begin{equation}
\label{non2}
(U\alpha) AY+\beta(\nabla_{Y}A)\xi+\beta A\phi AY=0.
\end{equation}
Let $X=\xi$, $Y\in\Gamma(D_U)$, $W\in\Gamma(TM)$ in (\ref{2}). Then by the fact that $Y\alpha=0$, we have
\begin{equation}
\label{non4}
(\xi\alpha) AY+c\phi AY+(\alpha c+k)\phi Y+\alpha(\nabla_YA)\xi+\alpha A\phi AY=0.
\end{equation}
From (\ref{non2}) and (\ref{non4}), we have
\begin{equation}
\label{non5}
(\xi\alpha-\frac{\alpha}{\beta}U\alpha) AY+c\phi AY+(\alpha c+k)\phi Y=0.
\end{equation}
Let $Y$ be unit vector field with $AY=\lambda Y$ and take inner product with $\phi Y$ in (\ref{non5}). Then we obtain
\begin{equation}
\label{non6}
c\lambda+\alpha c+k=0.
\end{equation}
Therefore,
\begin{equation}
\label{new}
AX=\lambda X, (\lambda=-\dfrac{\alpha c+k}{c})
\end{equation}
for any $X\in\Gamma(D_U)$.
On the other hand, if we put $X=Y$ in (\ref{9}), then we have
\begin{equation}
\alpha\lambda^2+c\lambda+k=0.
\label{15}
\end{equation}
From (\ref{non6}) and (\ref{15}) we have
\begin{equation}
\lambda^2=c.
\label{non7}
\end{equation}
Hence $c=1$ and $\lambda=\pm 1$. It follows that $\phi U\lambda=0$. So by Lemma~\ref{lemma1}, (\ref{10}), (\ref{15}) and (\ref{non7}), we obtain $\alpha=-\lambda(=\mp 1)$ and $k=0$. By substituting all these quantities into (\ref{10}), (\ref{11}) and (\ref{new}), we obtain a contradiction according to Lemma~\ref{lemma2}. This completes the proof.
 \end{proof}

\begin{proof}[Proof of Theorem~\ref{th0}]
($\Leftarrow$). It follows directly from Proposition~\ref{prop101} and Proposition~\ref{105}.

($\Rightarrow$). Suppose $M$ satisfies (\ref{1}). By Proposition~\ref{103} and
 Proposition~\ref{102}, $M$ is either totally $\eta$-umbilical or a Hopf hypersurface
 with $\alpha=k=1$, $c=-1$. Furthermore, by Proposition~\ref{prop101}, we can see that $k\neq 0$.
\end{proof}

\begin{proof}[Proof of Theorem~\ref{th3}]
($\Leftarrow$). It has been proved in Proposition~\ref{prop101}.

($\Rightarrow$). Suppose $M$ is a real hypersurface satisfying (\ref{umb}).
 Then $M$ also satisfies (\ref{1}). Hence from Theorem~\ref{th0}, $M$ is either totally
  $\eta$-umbilical or a Hopf hypersurface with $\alpha=k=1$, $c=-1$. However, by Proposition~\ref{105},
   the latter case cannot occur. We conclude that $M$ is totally $\eta$-umbilical, so it is locally congruent
    to one of the real hypersurfaces listed in Theorem~\ref{classifyumb}.
\end{proof}

\begin{proof}[Proof of Theorem~\ref{th2}]
Suppose the Jacobi operator is of Codazzi type. Then $M$ satisfies
(\ref{1}) with $k=0$. This contradicts the fact that $k\neq 0$ as stated
 in Theorem~\ref{th0}. The proof is completed.
\end{proof}




Email:

S.H.Kon: shkon@um.edu.my,

Tee-How Loo: looth@um.edu.my,

Shiquan Ren: renshiquan@gmail.com


\begin{thebibliography}{99}

\bibitem{3}
J. Berndt, Real hypersurfaces with constant principal curvatures in complex hyperbolic space,
\emph{J. Reine Angew. Math.} {\bf  395} (1989), 132-141.

\bibitem{focal}
T.E. Cecil and P.J. Ryan, Focal sets and real hypersurfaces in complex projective
 space, \emph{Trans. Amer. Math. Soc.} {\bf 269 }(1982), 481-499.

\bibitem{nihonkai}
Y.W. Choe, Characterizations of certain real hypersurfaces of a complex space form,
\emph{Nihonkai Math. J.} {\bf 6 }(1995), 97-114.

\bibitem{ki}
U.H. Ki and H. Kurihara, Real hypersurfaces with cyclic-parallel structure
 Jacobi operators in a non-flat complex space form,
    \emph{Bull. Aust. Math. Soc.} {\bf 81} (2010), 260-273.

\bibitem{34}
 M. Kimura, Real hypersurfaces and complex submanifolds in complex projective space,
\emph{Trans. Amer. Math. Soc.} {\bf  296:1} (1986), 137-149.

\bibitem{ricci}
 M. Kimura and S. Maeda, Characterizations of geodesic hyperspheres in a
  complex projective space in terms of Ricci tensors,
\emph{Yokohama Math. J.} {\bf  40} (1992), 35-43.

\bibitem{jdg}
M. Kon, Pseudo-Einstein real hypersurfaces in complex space forms,
\emph{J. Differential Geom.} {\bf 14} (1979), 339-354.

\bibitem{1}
    S.H. Kon, T.H. Loo and S. Ren, Real hypersurfaces in a complex space
    form with a condition on the structure Jacobi operator.
    \emph{Math. Slovaca}, to appear.

 \bibitem{maeda}
 Y. Maeda, On real hypersurfaces of a complex projective space,
  \emph{J. Math. Soc. Japan} {\bf  28:3} (1976), 529-540.

\bibitem{montiel}
S. Montiel, Real hypersurfaces of a complex hyperbolic space,
\emph{J. Math. Soc. Japan} {\bf  37:3} (1985), 515-535.

\bibitem{montiel-romero}
S. Montiel and A. Romero, On some real hypersurfaces of a complex hyperbolic space,
\emph{Geom. Dedicata} {\bf  20} (1986), 245-261.

\bibitem{Ryan}
R. Niebergall and P.J. Ryan, Real hypersurfaces in complex space forms,
    \emph{Tight and taut submanifolds MSRI Publication} {\bf 32} (1997), 233-304.

\bibitem{okumura}
M. Okumura, Certain almost contact hypersurfaces in Euclidean spaces,
\emph{Kodai Math. Sem. Rep.} {\bf  16} (1964), 44-54.

\bibitem{okumura2}
M. Okumura, Contact hypersurfaces in certain Kaehlerian manifolds,
\emph{Tohoku Math. J.} {\bf  18} (1966), 74-102.

\bibitem{okumura3}
M. Okumura, On some real hypersurfaces of a complex projective space,
\emph{Trans. Amer. Math. Soc.} {\bf 212} (1975), 355-364.

\bibitem{rocky}
M. Ortega, J.D. P\'{e}rez and F.G. Santos, Non-existence of real
hypersurfaces with parallel structure Jacobi operator in non-flat
complex space forms.
    \emph{Rocky Mountain J. Math.} {\bf 36} (2006), 1603-1613.


\bibitem{lie}
J.D. P\'{e}rez, F.G. Santos and Y.J. Suh, Real hypersurfaces in complex projective space
whose structure Jacobi operator is Lie $\xi$-parallel,
\emph{Differential Geom. Appl.} {\bf  22} (2005), 181-188.

\bibitem{codazzi}
J.D. P\'{e}rez, F.G. Santos and Y.J. Suh, Real hypersurfaces in complex projective
 space whose structure Jacobi operator is of Codazzi type.
    \emph{Canad. Math. Bull.} {\bf 50:3} (2007), 347-355.

\bibitem{vernon}
M.H. Vernon, Contact hypersurfaces of a complex hyperbolic space,
\emph{Tohoku Math. J.} {\bf  39} (1987), 215-222.



















\end{thebibliography}
\end{document}